\documentclass{article}%
\usepackage{amsmath}
\usepackage{amsfonts}
\usepackage{amssymb}
\usepackage{graphicx}%
\providecommand{\U}[1]{\protect\rule{.1in}{.1in}}
\newtheorem{theorem}{Theorem}[section]

\newtheorem{conjecture}[theorem]{Conjecture}
\newtheorem{corollary}[theorem]{Corollary}

\newtheorem{problem}[theorem]{Problem}
\newtheorem{proposition}[theorem]{Proposition}

\newenvironment{proof}[1][Proof]{\noindent\textbf{#1.} }{\ \rule{0.5em}{0.5em}}
\begin{document}

\author{Vadim E. Levit\\Department of Computer Science\\Ariel University, Israel\\levitv@ariel.ac.il
\and Eugen Mandrescu\\Department of Computer Science\\Holon Institute of Technology, Israel\\eugen\_m@hit.ac.il}
\title{The Roller-Coaster Conjecture Revisited}
\date{}
\maketitle

\begin{abstract}
A graph is \textit{well-covered} if all its maximal independent sets are of
the same cardinality \cite{plum}. If $G$ is a well-covered graph, has at least
two vertices, and $G-v$ is well-covered for every vertex $v$, then $G$ is a
$1$-\textit{well-covered graph} \cite{StaplesThesis}. We call $G$ a $\lambda
$-\textit{quasi-regularizable graph} if $\lambda\cdot\left\vert S\right\vert
\leq\left\vert N\left(  S\right)  \right\vert $ for every independent set $S$
of $G$. The \textit{independence polynomial} $I(G;x)$ is the generating
function of independent sets in a graph $G$ \cite{GutHar}.

The Roller-Coaster Conjecture \cite{Michael Travis}, saying that for every
permutation $\sigma$ of the set $\{\left\lceil \frac{\alpha}{2}\right\rceil
,...,\alpha\}$ there exists a well-covered graph $G$ with independence number
$\alpha$ such that the coefficients $\left(  s_{k}\right)  $ of $I(G;x)\ $%
satisfy
\[
s_{\sigma(\left\lceil \frac{\alpha}{2}\right\rceil )}<s_{\sigma(\left\lceil
\frac{\alpha}{2}\right\rceil +1)}<\cdots<s_{\sigma(\alpha)},
\]
has been validated in \cite{CuPe2017}.

In this paper we show that independence polynomials of $\lambda$%
-quasi-regularizable graphs are partially unimodal. More precisely, the
coefficients of an upper part of $I(G;x)$ are in non-increasing order. Based
on this finding, we prove that the domain of the Roller-Coaster Conjecture can
be shortened up to:
\[
\{\left\lceil \frac{\alpha}{2}\right\rceil ,\left\lfloor \frac{\alpha}%
{2}\right\rfloor +1,...,\min\left\{  \alpha,\left\lceil \frac{n-1}%
{3}\right\rceil \right\}  \}
\]
for well-covered graphs, and up to
\[
\{\left\lceil \frac{2\alpha}{3}\right\rceil ,\left\lceil \frac{2\alpha}%
{3}\right\rceil +1,...,\min\left\{  \alpha,\left\lceil \frac{n-1}%
{3}\right\rceil \right\}  \}
\]
for $1$-well-covered graphs, where $\alpha$\ stands for the independence
number, and $n$ is the cardinality of the vertex set.


\end{abstract}

\section{Introduction}

Throughout this paper $G=(V,E)$ is a simple (i.e., a finite, undirected,
loopless and without multiple edges) graph with vertex set $V=V(G)\neq
\emptyset$ and edge set $E=E(G)$. If $X\subset V$, then $G[X]$ is the subgraph
of $G$ induced by $X$. By $G-W$ we mean the subgraph $G[V-W]$, if $W\subset
V(G)$. We also denote by $G-F$ the subgraph of $G$ obtained by deleting the
edges of $F$, for $F\subset E(G)$, and we write shortly $G-e$, whenever $F$
$=\{e\}$.

The \textit{neighborhood} $N(v)$ of $v\in V\left(  G\right)  $ is the set
$\{w:w\in V\left(  G\right)  $ \textit{and} $vw\in E\left(  G\right)  \}$,
while the \textit{closed neighborhood} $N[v]$\ of $v$ is the set
$N(v)\cup\{v\}$. The \textit{neighborhood} $N(A)$ of $A\subseteq V\left(
G\right)  $ is $\{v\in V\left(  G\right)  :N(v)\cap A\neq\emptyset\}$, and
$N[A]=N(A)\cup A$.

$C_{n},K_{n},P_{n}$ denote respectively, the cycle on $n\geq3$ vertices, the
complete graph on $n\geq1$ vertices, and the path on $n\geq1$ vertices.

The \textit{disjoint union} of the graphs $G_{1},G_{2}$ is the graph
$G_{1}\cup G_{2}$ having the disjoint unions $V(G_{1})\cup V(G_{2})$ and
$E(G_{1})\cup E(G_{2})$ as a vertex set and an edge set, respectively. In
particular, $nG$ denotes the disjoint union of $n>1$ copies of the graph $G$.

An \textit{independent} set in $G$ is a set of pairwise non-adjacent vertices.
An independent set of maximum size is a \textit{maximum independent set} of
$G$, and the \textit{independence number }of $G$, denoted $\alpha(G)$, is the
cardinality of a maximum independent set in $G$.

A graph is \textit{well-covered} if all its maximal independent sets are of
the same size \cite{plum}. If $G$ is well-covered, without isolated vertices,
and $\left\vert V\left(  G\right)  \right\vert =2\alpha\left(  G\right)  $,
then $G$ is a \textit{very well-covered graph} \cite{Favaron1982}. The only
well-covered cycles are $C_{3}$, $C_{4}$, $C_{5}$ and $C_{7}$, while $C_{4}$
is the only very well-covered cycle. A well-covered graph (with at least two
vertices) is $1$-\textit{well-covered} if the deletion of every vertex of the
graph leaves a graph, which is well-covered as well \cite{StaplesThesis}. For
instance, $K_{2}$ is $1$-well-covered, while $P_{4}$ is very well-covered, but
not $1$-well-covered. Notice that $C_{7}$ is well-covered but not
$1$-well-covered. The only $1$-well-covered cycles are $C_{3}$ and $C_{5}$. A
graph $G$ belongs to \textit{class} $W_{2}$ if every two disjoint independent
sets in $G$ are contained in two disjoint maximum independent sets
\cite{StaplesThesis,Staples1979}. Clearly, $W_{1}\supseteq W_{2}$, where
$W_{1}$ is the family of all well-covered graphs.

\begin{theorem}
\cite{StaplesThesis}\label{th4} Let $G$ have no isolated vertices. Then $G$ is
$1$-well-covered if and only if $\ G$ belongs to the class $\mathbf{W}_{2}$.
\end{theorem}

If $G$ has an isolated vertex, then it is contained in all maximum independent
sets, and hence $G$ cannot be in class $\mathbf{W}_{2}$. However, a graph
having isolated vertices may be $1$-well-covered; e.g., $K_{3}\cup K_{1}$.

\begin{theorem}
\cite{LevMan2016b}\label{th1} Let $G$ be a graph without isolated vertices.
Then $G$ is $1$-well-covered if and only if for each non-maximum independent
set $A$ there are at least two disjoint independent sets $B_{1},B_{2}$ such
that $A\cup B_{1},A\cup B_{2}$ are maximum independent sets in $G$.
\end{theorem}

Let $s_{k}$ be the number of independent sets of size $k$ in a graph $G$. The
polynomial
\[
I(G;x)=s_{0}+s_{1}x+s_{2}x^{2}+\cdots+s_{\alpha}x^{\alpha},\quad\alpha
=\alpha\left(  G\right)  ,
\]
is called the \textit{independence polynomial} of $G$ \cite{GutHar}. For a
survey on independence polynomials of graphs see \cite{LevManGreece}. Closed
formulae for $I(G;x)$ of several families of graphs one can find in
\cite{LevMan2012,Zhu}, while some factorizations of independence polynomials
for certain classes of graphs are given in \cite{Wang}.

A polynomial is called unimodal if the sequence $(a_{0},a_{1},a_{2}%
,...,a_{n})$ of its coefficients is \textit{unimodal}, i.e., if there exists
an index $k\in\{0,1,...,n\}$, such that
\[
a_{0}\leq\cdots\leq a_{k-1}\leq a_{k}\geq a_{k+1}\geq\cdots\geq a_{n}.
\]
In \cite{AlMalSchErdos} it is proved that for every permutation $\sigma$ of
$\{1,2,...,\alpha\}$ there is a graph $G$ with $\alpha(G)=\alpha$ such that
the coefficients of $I(G;x)\ $satisfy $s_{\sigma(1)}<s_{\sigma(2)}%
<...<s_{\sigma(\alpha)}$.

\begin{theorem}
\cite{LevMan2008,Michael Travis}\label{th5} If $G$ is a well-covered graph,
then $s_{0}\leq s_{1}\leq\cdots\leq s_{\left\lceil \frac{\alpha(G)}%
{2}\right\rceil }$.
\end{theorem}

Several results concerning the independence polynomials of well-covered graphs
are presented in \cite{Brown,LevMan2003a,LevMan2003b,LevMan2004}. It is known
that there exist well-covered graphs whose independence polynomials are not
unimodal \cite{LevMan2006b,Michael Travis}.

\begin{conjecture}
[Roller-Coaster Conjecture]\cite{Michael Travis} For every permutation
$\sigma$ of the set $\{\left\lceil \frac{\alpha}{2}\right\rceil ,...,\alpha\}$
there is a well-covered graph $G$ with $\alpha(G)=\alpha$ such that the
coefficients of $I(G;x)\ $satisfy $s_{\sigma(\left\lceil \frac{\alpha}%
{2}\right\rceil )}<s_{\sigma(\left\lceil \frac{\alpha}{2}\right\rceil
+1)}<\cdots<s_{\sigma(\alpha)}$.
\end{conjecture}

The Roller-Coaster Conjecture has been verified for well-covered graphs $G$
having $\alpha(G)$ $\leq7$ \cite{Michael Travis}, and later for $\alpha(G)$
$\leq11$ \cite{Matchett}. In the case of very well-covered graphs, the domain
of the Roller-Coaster Conjecture can be shortened to $\{\left\lceil
\frac{\alpha}{2}\right\rceil ,\left\lceil \frac{\alpha}{2}\right\rceil
+1,...,\left\lceil \frac{2\alpha-1}{3}\right\rceil \}$, where $\alpha$ stands
for the independence number \cite{LevMan2006a}. Recently, the Roller-Coaster
Conjecture was validated in \cite{CuPe2017}.

In this paper we show that the domain of the Roller-Coaster Conjecture can be
shortened to:

\begin{itemize}
\item $\{\left\lceil \frac{\alpha}{2}\right\rceil ,\left\lfloor \frac{\alpha
}{2}\right\rfloor +1,...,\min\left\{  \alpha,\left\lceil \frac{n-1}%
{3}\right\rceil \right\}  \}$ for well-covered graphs of order $n$;

\item $\{\left\lceil \frac{2\alpha}{3}\right\rceil ,\left\lceil \frac{2\alpha
}{3}\right\rceil +1,...,\min\left\{  \alpha,\left\lceil \frac{n-1}%
{3}\right\rceil \right\}  \}$ for $1$-well-covered graphs of order $n$.
\end{itemize}

Actually, $\min\left\{  \alpha,\left\lceil \frac{n-1}{3}\right\rceil \right\}
<\alpha$ only for $n\leq3\alpha-2$. It means that one may formulate an
overhauled Roller-Coaster Conjecture as follows.

\begin{conjecture}
Let $\alpha\geq2$ and $n\geq4$ be integers satisfying $2\alpha\leq
n\leq3\alpha-2$. Then for every permutation $\sigma$ of the set $\{\left\lceil
\frac{\alpha}{2}\right\rceil ,\left\lceil \frac{\alpha}{2}\right\rceil
+1,...,\left\lceil \frac{n-1}{3}\right\rceil \}$ there exists a well-covered
graph $G$ with $\alpha(G)=\alpha$ and $\left\vert V(G)\right\vert =n$ such
that the coefficients of $I(G;x)\ $satisfy $s_{\sigma(\left\lceil \frac
{\alpha}{2}\right\rceil )}<s_{\sigma(\left\lceil \frac{\alpha}{2}\right\rceil
+1)}<\cdots<s_{\sigma(\left\lceil \frac{n-1}{3}\right\rceil )}$.
\end{conjecture}

\section{Results}

We call $G$ a $\lambda$-\textit{quasi-regularizable graph} if $\lambda>0$ and
$\lambda\cdot\left\vert S\right\vert \leq\left\vert N\left(  S\right)
\right\vert $ is true for every independent set $S$ of $G$. If $\lambda=1$,
then $G$ is a \textit{quasi-regularizable} graph \cite{Berge1982}.

For a graph $G$ and $1\leq k<\alpha(G)$, let $\Omega_{k}\left(  G\right)
=\{W:\left\vert W\right\vert =k,W$ \textit{is independent in }$G\}$ and
$H_{k}\left(  G\right)  =(\Omega_{k}\left(  G\right)  ,\Omega_{k+1}\left(
G\right)  ,Y)$ be the bipartite graph with bipartition $\left\{  \Omega
_{k}\left(  G\right)  ,\Omega_{k+1}\left(  G\right)  \right\}  $ and such that
$WU\in Y$ if and only if $W\subset U$. It is clear that $\left\vert \Omega
_{k}\right\vert =s_{k}$.

\begin{theorem}
\label{th13}If $G$ is a $\lambda$-quasi-regularizable graph of order $n$, then
the following assertions are true:

\emph{(i) }$(k+1)\cdot s_{k+1}\leq(n-\left(  \lambda+1\right)  k)\cdot
s_{k},0\leq k<\alpha\left(  G\right)  $;

\emph{(ii) }$s_{r}\geq s_{r+1}\geq\cdots\geq s_{\alpha\left(  G\right)  }$,
for $r=\left\lceil \frac{n-1}{\lambda+2}\right\rceil .$
\end{theorem}

\begin{proof}
Every $U\in\Omega_{k+1}\left(  G\right)  $ has $k+1$ subsets in $\Omega
_{k}\left(  G\right)  $, which means that the degree of every vertex $U$ in
$H_{k}\left(  G\right)  $ is equal to $k+1$. Consequently, we obtain
\[
\left\vert Y\right\vert =(k+1)\cdot\left\vert \Omega_{k+1}\left(  G\right)
\right\vert =(k+1)\cdot s_{k+1}.
\]

Every $W\in\Omega_{k}\left(  G\right)  $ may be extended to some $U\in
\Omega_{k+1}\left(  G\right)  $ by means of a vertex belonging to $V\left(
G\right)  -N\left[  W\right]  $. Since $G$ is a $\lambda$-quasi-regularizable,
we have
\[
\left\vert N[W]\right\vert =\left\vert W\cup N(W)\right\vert \geq\left(
\lambda+1\right)  \cdot\left\vert W\right\vert ,
\]
and hence,
\[
(k+1)\cdot s_{k+1}\leq\left(  n-\left(  \lambda+1\right)  k\right)  \cdot
s_{k}.
\]
Therefore, we get
\[
s_{k+1}\leq\frac{n-\left(  \lambda+1\right)  k}{k+1}\cdot s_{k},
\]
which implies $s_{k+1}\leq s_{k}$ for every $k$ satisfying
\[
\frac{n-\left(  \lambda+1\right)  \cdot k}{k+1}\leq1\Leftrightarrow
k\geq\left\lceil \frac{n-1}{\lambda+2}\right\rceil ,
\]
as claimed.
\end{proof}

In particular, for $\lambda=1$, we deduce the following.

\begin{corollary}
\label{cor3}Let $G$ be a quasi-regularizable graph of order $n\geq2$ with
$\alpha(G)=\alpha$. Then

\emph{(i)} $\left(  k+1\right)  \cdot s_{k+1}\leq\left(  n-2k\right)  \cdot
s_{k},1\leq k<\alpha$;

\emph{(ii) }$s_{\left\lceil \frac{n-1}{3}\right\rceil }\geq s_{\left\lceil
\frac{n-1}{3}\right\rceil +1}\geq\cdots\geq s_{\alpha}$.
\end{corollary}

Taking into account Theorem \ref{th5}, Corollary \ref{cor3}, and the fact that
every well-covered graph is quasi-regularizable \cite{Berge1982}, we arrive at
the following.

\begin{corollary}
\label{cor2}Let $G$ be a well-covered graph of order $n\geq2$ with
$\alpha(G)=\alpha$. Then

\emph{(i)} $\left(  \alpha-k\right)  \cdot s_{k}\leq\left(  k+1\right)  \cdot
s_{k+1}\leq\left(  n-2k\right)  \cdot s_{k},1\leq k<\alpha$;

\emph{(ii) }$s_{0}\leq s_{1}\leq\cdots\leq s_{\left\lceil \frac{\alpha}%
{2}\right\rceil }$ and $s_{\left\lceil \frac{n-1}{3}\right\rceil }\geq
s_{\left\lceil \frac{n-1}{3}\right\rceil +1}\geq\cdots\geq s_{\alpha}$.
\end{corollary}

Combining Theorem \ref{th5} and Corollary \ref{cor2}, we infer that for
well-covered graphs, the domain of the Roller-Coaster Conjecture can be
shortened to $\{\left\lceil \frac{\alpha}{2}\right\rceil ,\left\lceil
\frac{\alpha}{2}\right\rceil +1,...,\left\lceil \frac{n-1}{3}\right\rceil \}$,
whenever $2\leq\alpha$ and $4\leq n\leq3\alpha-2$.

Since each very well-covered graph is of order twice its independence number,
we obtain the following.

\begin{corollary}
\cite{LevMan2006a}\label{cor1} If $G$ is a very well-covered graph of order
$n\geq2$ with $\alpha(G)=\alpha$, then\emph{ }$s_{0}\leq s_{1}\leq\cdots\leq
s_{\left\lceil \frac{\alpha}{2}\right\rceil }$ and $s_{\left\lceil
\frac{2\alpha-1}{3}\right\rceil }\geq s_{\left\lceil \frac{2\alpha-1}%
{3}\right\rceil +1}\geq\cdots\geq s_{\alpha}$.
\end{corollary}

Clearly, $nK_{2}$ is $1$-well-covered for $n\geq1$, and has exactly
$2\alpha(G)$ vertices, while each graph $G\in$ $\left\{  C_{5}\cup
nK_{2},C_{3}\cup nK_{2}\right\}  ,n\geq1$, is $1$-well-covered and has exactly
$2\alpha(G)+1$ vertices. One can show that $C_{3}$ and $C_{5}$ are the only
connected $1$-well-covered graphs with exactly $2\alpha(G)+1$ vertices
\cite{LevMan2016b}.

\begin{proposition}
\cite{LevMan2016b}\label{prop11} If a connected\textbf{ }graph $G\neq K_{2}$
is $1$-well-covered, then:

\emph{(i) }$G$ has at least $2\alpha(G)+1$ vertices;

\emph{(ii) }$\left\vert A\right\vert <\left\vert N\left(  A\right)
\right\vert $ for every independent set $A$.
\end{proposition}

Proposition \ref{prop11}\emph{(i)} implies that $K_{2}$ is the unique very
well-covered connected graph and also $1$-well-covered. In addition,
$I(K_{2};x)=1+2x$ is unimodal.

\begin{theorem}
\label{th3}If $G$ is a connected $1$-well-covered graph, $\left\vert V\left(
G\right)  \right\vert =n>2$, and $\alpha\left(  G\right)  =\alpha$, then the
following assertions are true:

\emph{(i)} $2\left(  \alpha-k\right)  \cdot s_{k}\leq(k+1)\cdot s_{k+1},1\leq
k<\alpha$;

\emph{(ii)} $s_{0}\leq s_{1}\leq\cdots\leq s_{\left\lceil \frac{2\alpha}%
{3}\right\rceil }$;

\emph{(iii) }$\left(  k+1\right)  \cdot s_{k+1}<\left(  n-2k\right)  \cdot
s_{k},1\leq k<\alpha$;

$\emph{(iv)\ }s_{\left\lceil \frac{n-1}{3}\right\rceil }>s_{\left\lceil
\frac{n-1}{3}\right\rceil +1}>\cdots>s_{\alpha}$.
\end{theorem}

\begin{proof}
\emph{(i)} According to Proposition \ref{prop11}\emph{(i)}, we have that
$2\alpha\cdot s_{0}=2\alpha\leq s_{1}=\left\vert V\left(  G\right)
\right\vert $.

Every $U\in\Omega_{k+1}\left(  G\right)  $ has $k+1$ subsets in $\Omega
_{k}\left(  G\right)  $, which means that the degree of every vertex $U$ in
$H$ is equal to $k+1$. Consequently, $\left\vert Y\right\vert =(k+1)\cdot
\left\vert \Omega_{k+1}\left(  G\right)  \right\vert =(k+1)\cdot s_{k+1}$. On
the other hand, by Theorem \ref{th1}, every $W\in\Omega_{k}\left(  G\right)  $
can be extended by two disjoint independent sets $B_{1},B_{2}$ such that
$W_{i}\cup B_{1},W_{i}\cup B_{2}$ are maximum independent sets in $G$. In
other words, the degree of every vertex $W\in\Omega_{k}\left(  G\right)  $ is
at least $2\left(  \alpha-k\right)  $.

In conclusion, we obtain $2\left(  \alpha-k\right)  \cdot s_{k}\leq(k+1)\cdot
s_{k+1}$, and this implies \emph{(i)}.

\emph{(ii)} According Part \emph{(i)}, we have
\[
s_{k}\leq\frac{k+1}{2\left(  \alpha-k\right)  }\cdot s_{k+1},
\]
which ensures that $s_{k}\leq s_{k+1}$ for every $k$ satisfying $\frac
{k+1}{2\left(  \alpha-k\right)  }\leq1$, i.e., for $k\leq\frac{2\alpha-1}{3}$,
at least. In other words, the monotone part of the sequence of coefficients
goes up to $k+1\leq\left\lfloor \frac{2\alpha+2}{3}\right\rfloor =\left\lceil
\frac{2\alpha}{3}\right\rceil $.

\emph{(iii)} and \emph{(iv) }By Proposition \ref{prop11}\emph{(ii)},
$\left\vert A\right\vert <\left\vert N\left(  A\right)  \right\vert $ for
every independent set $A$. To get the result, one has just to follow the lines
of the proof of Theorem \ref{th13} changing \textquotedblleft$\leq
$\textquotedblright\ for \textquotedblleft$<$\textquotedblright, when needed.
\end{proof}

In other words, for $1$-well-covered graphs, the domain of the Roller-Coaster
Conjecture can be shortened to $\{\left\lceil \frac{2\alpha}{3}\right\rceil
,\left\lceil \frac{2\alpha}{3}\right\rceil +1,...,\left\lceil \frac{n-1}%
{3}\right\rceil \}$, whenever $n\leq3\alpha-2$.

Let $\mathcal{H}=\{H_{v}:v\in V(G)\}$ be a family of graphs indexed by the
vertex set of a graph $G$. The corona $G\circ\mathcal{H}$ of $G$ and
$\mathcal{H}$ is the disjoint union of $G$ and $H_{v},v\in V(G)$, with
additional edges joining each vertex $v\in V(G)$ to all the vertices of
$H_{v}$. If $H_{v}=H$ for every $v\in V(G)$, then we denote $G\circ H$ instead
of $G\circ\mathcal{H}$ \cite{FruchtHarary}.

\begin{theorem}
\cite{LevMan2016b}\label{th2} Let $G$ be an arbitrary graph and $\mathcal{H}%
=\{H_{v}:v\in V(G)\}$ be a family of non-empty graphs. Then $G\circ
\mathcal{H}$ is $1$-well-covered if and only if each $H_{v}\in\mathcal{H}$ is
a complete graph of order two at least, for every non-isolated vertex $v$,
while for each isolated vertex $u$, its corresponding $H_{u}$ may be any
complete graph.
\end{theorem}

It is easy to see that $H\circ K_{1}$ is very well-covered for every graph
$H$, and some properties of $I\left(  H\circ K_{1};x\right)  $ are presented
in \cite{LevMan2008}. Several findings concerning the palindromicity of
$I\left(  H\circ Y;x\right)  $ are proved in
\cite{LevMan2007,LevMan2016a.St98,Zhu2016}.

\begin{theorem}
\label{th6}\cite{Gu92d} $I\left(  H\circ Y;x\right)  =\left(  I\left(
Y;x\right)  \right)  ^{n}\bullet I\left(  H;\frac{x}{I\left(  Y;x\right)
}\right)  $, where $n=\left\vert V\left(  H\right)  \right\vert $.
\end{theorem}

Theorem \ref{th6} allows finding closed formulae for $I\left(  H\circ
Y;x\right)  $, once such formulae are known for both $I\left(  H;x\right)  $
and $I\left(  Y;x\right)  $; for instance, one can obtain closed formulae for
$I\left(  H\circ K_{p};x\right)  $, where $H\in\left\{  P_{n},C_{n}%
,K_{1,n}\right\}  $ \cite{Arocha,GutHar,LevMan2008}.

\begin{theorem}
Let $H$\ be a connected graph. If $G=H\circ K_{2}$ and $\alpha(G)=\alpha$,
then the following assertions are true:

\emph{(i)} $G$ is a $1$-well-covered graph;

\emph{(ii)} $G$ is a $2$-quasi-regularizable graph of order $n=3\alpha$;

\emph{(iii)} $2\left(  \alpha-k\right)  \cdot s_{k}\leq\left(  k+1\right)
\cdot s_{k+1}\leq3\left(  \alpha-k\right)  \cdot s_{k},1\leq k<\alpha$;

\emph{(iv) }$s_{0}\leq s_{1}\leq\cdots\leq s_{\left\lceil \frac{2\alpha}%
{3}\right\rceil }$ and $s_{\left\lceil \frac{3\alpha-1}{4}\right\rceil }%
\geq\cdots\geq s_{\alpha-1}\geq s_{\alpha}$;

\emph{(v)} if $\alpha\geq3$, then $s_{\alpha-3}\cdot$ $s_{\alpha-1}\leq
s_{\alpha-2}^{2}$ and $s_{\alpha-2}\cdot$ $s_{\alpha}\leq s_{\alpha-1}^{2}$;

\emph{(vi)} if $\alpha\leq17$, then $I\left(  G;x\right)  $ is unimodal.
\end{theorem}

\begin{proof}
\textbf{\ }\emph{(i) }It follows from Theorem \ref{th2}.

\emph{(ii) }Let $S=S_{1}\cup S_{2}$ be an independent set in $G$, where
$S_{1}\subseteq V(H)$, while $S_{2}\subseteq V(G)-V(H)$. Then $2\left\vert
S_{1}\right\vert =\left\vert N_{G}(S_{1})-V(H)\right\vert \leq\left\vert
N_{G}(S_{1})\right\vert $, because every vertex of $S_{1}$ has exactly two
neighbors in $V(G)-V(H)$, and $2\left\vert S_{2}\right\vert =\left\vert
N_{G}(S_{2})\right\vert $, since each vertex from $S_{2}$ has exactly two
neighbors in $G$. Hence, we get that:
\[
2\left\vert S\right\vert =2\left\vert S_{1}\right\vert +2\left\vert
S_{2}\right\vert \leq\left\vert N_{G}(S_{1})-V(H)\right\vert +\left\vert
N_{G}(S_{2})\right\vert \leq\left\vert N_{G}(S)\right\vert ,
\]
i.e., $G$ is $2$-quasi-regularizable. Clearly, $\alpha=\left\vert
V(H)\right\vert $. Thus $n=3\alpha$.

\emph{(iii)} It follows from Theorem \ref{th3}\emph{(i)}, Theorem
\ref{th13}\emph{(i)}, and the fact that $n=3\alpha$.

\emph{(iv) }By Theorem \ref{th13}, $s_{\left\lceil \frac{3\alpha-1}%
{4}\right\rceil }\geq\cdots\geq s_{\alpha-1}\geq s_{\alpha}$, because $G$ is
$2$-quasi-regularizable. According to Theorem \ref{th3}\emph{(ii)}, the
polynomial $I(G\circ K_{2};x)$ satisfies $s_{0}\leq s_{1}\leq\cdots\leq
s_{\left\lceil \frac{2\alpha}{3}\right\rceil }$.

\emph{(v)} Let us specialize the inequality $2\left(  \alpha-k\right)  \cdot
s_{k}\leq\left(  k+1\right)  \cdot s_{k+1}$ at $k=\alpha-3$ and the inequality
$\left(  k+1\right)  \cdot s_{k+1}\leq3\left(  \alpha-k\right)  \cdot s_{k}$
at $k=\alpha-2$. It implies
\[
s_{\alpha-3}\cdot s_{\alpha-1}\leq\frac{\left(  \alpha-2\right)  }{\left(
\alpha-1\right)  }\cdot s_{\alpha-2}^{2}\leq s_{\alpha-2}^{2}%
\]

When we substitute $k=\alpha-2$ and $k=\alpha-1$ in the same manner, we
obtain
\[
s_{\alpha-2}\cdot s_{\alpha}\leq\frac{3\left(  \alpha-1\right)  }{4\alpha
}\cdot s_{\alpha-1}^{2}\leq s_{\alpha-1}^{2}\text{.}%
\]

\emph{(vi)} By part \emph{(iv)}, the sequence of coefficients of $I\left(
G;x\right)  $ is non-decreasing up to $\left\lceil \frac{2\alpha}%
{3}\right\rceil $ and non-increasing starting from $\left\lceil \frac
{3\alpha-1}{4}\right\rceil $. In addition, the constraint $\alpha\leq17$
ensures that $\left\lceil \frac{3\alpha-1}{4}\right\rceil -\left\lceil
\frac{2\alpha}{3}\right\rceil \leq1$.
\end{proof}

In other words, if $G$ can be represented as $H\circ K_{2}$, then $G$ is
$1$-well-covered and the domain of the Roller-Coaster Conjecture can be
shortened to $\{\left\lceil \frac{2\alpha}{3}\right\rceil ,\left\lceil
\frac{2\alpha}{3}\right\rceil +1,...,\left\lceil \frac{3\alpha-1}%
{4}\right\rceil \}$.

It is known that:

\begin{itemize}
\item each polynomial with positive coefficients that has only real roots is unimodal;

\item there exist graphs whose independence polynomials have all the roots
real (for example, $K_{1,3}$-free graphs \cite{ChudSey}, $P_{n}\circ
K_{1\text{ }}$for any $n\geq1$ \cite{LevMan2008});

\item $I\left(  H\circ K_{p};x\right)  $ has only real roots if and only if
the same is true for $I\left(  H;x\right)  $ \cite{LevMan2008,Man2009}.
\end{itemize}

Hence, using Theorem \ref{th2}, we get the following.

\begin{corollary}
\label{cor4}If $I\left(  H;x\right)  $ has only real roots and $p\geq2$, then
every graph
\[
G\in\left\{  H\circ K_{p},\left(  H\circ K_{p}\right)  \circ K_{p},\left(
\left(  H\circ K_{p}\right)  \circ K_{p}\right)  \circ K_{p},...\right\}
\]
is $1$-well-covered and its $I\left(  G;x\right)  $ is unimodal, as having all
its roots real.
\end{corollary}

\section{Conclusions and future work}

In this paper we proved that for $1$-well-covered graphs the\textit{\ }%
\textquotedblleft chaotic interval\textquotedblright\ $(\left\lceil
\frac{\alpha}{2}\right\rceil ,\left\lceil \frac{\alpha}{2}\right\rceil
+1,...,\alpha)$ involved in Roller-Coaster Conjecture can be shortened to
$\{\left\lceil \frac{2\alpha}{3}\right\rceil ,\left\lceil \frac{2\alpha}%
{3}\right\rceil +1,...,\alpha\}$. Based on this finding, we propose a
Roller-Coaster Conjecture for $1$-well-covered graphs as follows.

\begin{conjecture}
\label{con1}For every permutation $\sigma$ of the set $\{\left\lceil
\frac{2\alpha}{3}\right\rceil ,\left\lceil \frac{2\alpha}{3}\right\rceil
+1,...,\alpha\}$ there exists a $1$-well-covered graph $G$ with $\alpha
(G)=\alpha$ and $\left\vert V(G)\right\vert =n$ such that the coefficients of
$I(G;x)\ $satisfy $s_{\sigma(\left\lceil \frac{2\alpha}{3}\right\rceil
)}<s_{\sigma(\left\lceil \frac{2\alpha}{3}\right\rceil +1)}<\cdots
<s_{\sigma(\alpha)}$.
\end{conjecture}

We incline to think that Conjecture \ref{con1} can be validated using a
technique similar to one presented in \cite{CuPe2017}. The only obstacle we
see now is in constructing a $1$-well-covered graph $G$\ such that for every
given positive integer $k$ each $S\in\Omega_{k+1}\left(  G\right)  $ is
included in exactly two maximum independent sets.

\begin{problem}
Characterize $1$-well-covered graphs whose independence polynomials are unimodal.
\end{problem}

The nature and location of the roots of $I\left(  G;x\right)  $ for a
well-covered graph $G$ were first analyzed in \cite{Brown}. It is worth
mentioning that there are $1$-well-covered graphs whose independence
polynomials have non-real roots; e.g., $I(K_{1,3}\circ K_{2};x)=1+12x+51x^{2}%
+93x^{3}+62x^{4}$. Taking into account Corollary \ref{cor4}, we propose the following.

\begin{problem}
Characterize $1$-well-covered graphs whose independence polynomials have all
the roots real.
\end{problem}

\end{document}